\documentclass[leqno, a4paper, 12pt]{article}

\usepackage{amssymb, amscd, amsthm}

\usepackage[all]{xy}

\theoremstyle{plain}
\newtheorem{theorem}{Theorem}
\newtheorem{lemma}[theorem]{Lemma}

\newtheorem{corollary}[theorem]{Corollary}

\theoremstyle{definition}

\newtheorem{remark}[theorem]{Remark}

\newtheorem{example}[theorem]{Example}

\newcommand\bP{{\mathbb P}}

\newcommand\cL{{\mathcal L}}

\newcommand\cO{{\mathcal O}}

\newcommand\fm{\mathfrak{m}}

\newcommand\ant{{\rm ant}}

\newcommand\diag{{\rm diag}}

\newcommand\id{{\rm id}}

\newcommand\red{{\rm red}}

\newcommand\Aut{{\rm Aut}}

\newcommand\End{{\rm End}}

\newcommand\GL{{\rm GL}}

\newcommand\Spec{{\rm Spec}}

\newcommand\Sym{{\rm Sym}}

\title{Epimorphic subgroups of algebraic groups}

\author{Michel Brion}

\date{}

\begin{document}

\maketitle

\begin{abstract}
In this note, we show that the epimorphic subgroups of an algebraic group are 
exactly the pull-backs of the epimorphic subgroups of its affinization. We also 
obtain epimorphicity criteria for subgroups of affine algebraic groups, which 
generalize a result of Bien and Borel. Moreover, we extend the affinization
theorem for algebraic groups to homogeneous spaces.
\end{abstract}

\section{Introduction and statement of the results}
\label{sec:intr}

The algebraic groups considered in this note are the group schemes of 
finite type over a field $k$. They form the objects of a category, 
with morphisms being the homomorphisms of $k$-group schemes. 
One of the most basic questions one may ask about this category 
is to describe \emph{monomorphisms} and \emph{epimorphisms}.
Recall that a morphism $f: G \to H$ is a monomorphism if it satisfies 
the left cancellation property: for any algebraic group $G'$ and for
any morphisms $f_1,f_2: G' \to G$ such that $f \circ f_1 = f \circ f_2$, 
we have $f_1 = f_2$. Likewise, $f$ is an epimorphism if it satisfies the 
right cancellation property.

The answer to this question is very easy and well-known for monomorphisms:
these are exactly the homomorphisms with trivial (scheme-theoretic) 
kernel, or equivalently the closed immersions of algebraic groups. Also, 
$f: G \to H$ is an epimorphism if and only if so is the inclusion of its 
scheme-theoretic image. This reduces the description of epimorphisms to 
that of the \emph{epimorphic subgroups} of an algebraic group $G$, i.e., 
of those algebraic subgroups $H$ such that every morphism $G \to G'$ 
is uniquely determined by its pull-back to $H$. The purpose of this note
is to characterize such subgroups.

Examples of epimorphic subgroups include the parabolic subgroups of 
a smooth connected affine algebraic group $G$, i.e., the algebraic 
subgroups $H \subset G$ such that the homogeneous space $G/H$ is proper, 
or equivalently projective. Indeed, for any morphisms 
$f_1, f_2: G \to G'$ which coincide on $H$, the map 
$G \to G'$, $x \mapsto f_1(x) \, f_2(x^{-1})$ factors through 
a map $\varphi : G/H \to G'$. But every such morphism 
is constant: to see this, we may assume $k$ algebraically closed;
then $G/H$ is covered by rational curves (images of morphisms 
$\bP^1 \to G/H$), while every morphism $\bP^1 \to G'$ is constant.

In the category of smooth connected affine algebraic groups over
an algebraically closed field, the epimorphic subgroups have been studied 
by Bien and Borel in \cite{Bien-Borel-I, Bien-Borel-II}; see also 
\cite[\S 23]{Grosshans} for a more detailed exposition, and 
\cite[\S 4]{Bien-Borel-Kollar} for further developments. In particular,
\cite[Thm.~1]{Bien-Borel-I} presents several epimorphicity criteria 
in that setting. Our first result extends most of these criteria to
affine algebraic groups over an arbitrary field. To state it, 
let us define \emph{affine epimorphic subgroups} of an affine algebraic 
group $G$ as those algebraic subgroups $H \subset G$ that are epimorphic
in the category of affine algebraic groups. (Clearly, epimorphic implies
affine epimorphic. In fact, the converse holds, as we will show
in Corollary \ref{cor:aff}).  

\begin{theorem}\label{thm:BB}
The following conditions are equivalent for an algebraic subgroup $H$ 
of an affine algebraic group $G$:

\begin{enumerate}

\item[{\rm (i)}] $H$ is affine epimorphic in $G$.

\item[{\rm (ii)}] $\cO(G/H) = k$.

\item[{\rm (iii)}] For any finite-dimensional $G$-module $V$, 
we have the equality of fixed point subschemes $V^H = V^G$.

\item[{\rm (iv)}] For any finite-dimensional $G$-module $V$, if 
$V = V_1 \oplus V_2$, where $V_1,V_2$ are $H$-sub\-modules, then
$V_1,V_2$ are $G$-submodules.

\end{enumerate}

\end{theorem}

This result is proved in Section \ref{sec:BB} by adapting the argument 
of \cite[Thm.~1]{Bien-Borel-I} (see also \cite[Thm.~13]{Nguyen-Dao}). 
When $G$ is smooth and connected, condition (ii) is equivalent to 
the $k$-vector space $\cO(G/H)$ being finite-dimensional. But this fails
for non-connected groups (just take $H$ to be the trivial subgroup of
a non-trivial finite group $G$) and for non-smooth groups as well
(take $G$, $H$ as above with $G$ infinitesimal).

For the category of finite-dimensional Lie algebras over a field 
of characteristic $0$, the equivalence of conditions (i), (iii) and (iv) 
has been obtained by Bergman in an unpublished manuscript which was the
starting point of \cite{Bien-Borel-I}; see \cite[Cor.~3.2]{Bergman}, and
\cite{Panyushev} for recent developments based on the (related but not
identical) notion of wide subalgebra of a semi-simple Lie algebra.

Our second result yields an epimorphicity criterion in the category 
of all algebraic groups. To formulate it, recall the affinization theorem
(see \cite[\S III.3.8]{Demazure-Gabriel}): every algebraic group $G$ has 
a smallest normal algebraic subgroup $N$ such that the quotient $G/N$ 
is affine. Moreover, $N$ is smooth, connected, and contained in the center 
of the neutral component $G^0$. Also, $N$ is anti-affine (i.e., 
$\cO(N) = k$) and $N$ is the largest algebraic subgroup of $G$ satisfying 
this property; we denote $N$ by $G_{\ant}$. The quotient morphism 
$G \to G/G_{\ant}$ is the affinization morphism, i.e., the canonical map 
\[ \varphi_G : G \longrightarrow \Spec \, \cO(G). \] 
We may now state:

\begin{theorem}\label{thm:main}
The following conditions are equivalent for an algebraic subgroup 
$H$ of an algebraic group $G$:

\begin{enumerate}

\item[{\rm (i)}] $H$ is epimorphic in $G$.

\item[{\rm (ii)}] $H \supset G_{\ant}$ and $\cO(G/H) = k$.

\end{enumerate}

\end{theorem}

This result is proved in Section \ref{sec:main}, after gathering auxiliary results 
in Section \ref{sec:aux}. 

Note that the formations of $\cO(G/H)$ and $G_{\ant}$ commute with base 
change by field extensions of $k$. In view of Theorem \ref{thm:main},
this yields the first assertion of the following:

\begin{corollary}\label{cor:aff}
Let $G$ be an algebraic group, and $H$ an algebraic subgroup.

\begin{enumerate}

\item[{\rm (i)}] $H$ is epimorphic in $G$ if and only if the
base change $H_{k'}$ is epimorphic in $G_{k'}$ for some field
extension $k'$ of $k$.

\item[{\rm (ii)}] When $G$ is affine, $H$ is epimorphic in $G$
if and only if it is affine epimorphic.
 
\end{enumerate}

\end{corollary}

The second assertion follows readily from Theorems \ref{thm:BB}
and \ref{thm:main}. As a consequence, \emph{the epimorphic subgroups 
of an algebraic group $G$ are exactly the pull-backs of the epimorphic 
subgroups of its affinization.} 

Theorem \ref{thm:main} and Corollary \ref{cor:aff} reduce the description 
of the epimorphic subgroups of an algebraic group $G$ over $k$,
to the case where $G$ is affine and $k$ is algebraically closed. When
$G$ is smooth, our next result yields further reductions:

\begin{theorem}\label{thm:smooth}
The following conditions are equivalent for an algebraic subgroup $H$ 
of a smooth algebraic group $G$ over an algebraically closed field $k$:

\begin{enumerate}

\item[{\rm (i)}] $H$ is epimorphic in $G$.

\item[{\rm (ii)}] The reduced subgroup $H_{\red}$ is epimorphic in $G$.

\item[{\rm (iii)}] The reduced neutral component $H^0_{\red}$ is epimorphic 
in $G^0$ and the natural map $H_{\red}/H^0_{\red} \to G/G^0$ is surjective.

\end{enumerate}

\end{theorem}

This result is proved in Section \ref{sec:smooth}. 

In view of the above results, the class of homogeneous spaces 
$X = G/H$ such that $\cO(X) = k$ deserves further consideration. 
These anti-affine homogeneous spaces feature in an extension of 
the affinization theorem for algebraic groups (see 
\cite[III.3.8]{Demazure-Gabriel}), which is our final result. 
To state it, recall that a quasi-compact scheme $Z$ is said 
to be quasi-affine if the affinization map
$\varphi_Z : Z \to \Spec \, \cO(Z)$ is an open immersion (see
\cite[II.5.1.2]{EGA} for further characterizations).
\begin{theorem}\label{thm:aff}
Let $G$ be an algebraic group, and $H$ an algebraic subgroup. 

\begin{enumerate}

\item[{\rm (i)}] There exists a smallest algebraic subgroup $L$ 
of $G$ containing $H$ such that $G/L$ is quasi-affine. Moreover,
$\cO(G/L) = \cO(G/H)$ and the affinization map 
\[ \varphi_{G/H} : G/H \longrightarrow \Spec \, \cO(G/H) =: X \] 
is the composition
\[ G/H \stackrel{u}{\longrightarrow} G/L 
\stackrel{\varphi_{G/L}}{\longrightarrow} \Spec \, \cO(G/L) = X, \] 
where $u$ denotes the canonical morphism. 

\item[{\rm (ii)} ]The formation of $L$ commutes
with base change by arbitrary field extensions.

\item[{\rm (iii)}] $L$ is the largest subgroup of $G$ containing
$H$ such that $L/H$ is anti-affine.

\item[{\rm (iv)}]  $L/H$ is geometrically irreducible.

\item[{\rm (v)}] If $G$ is affine, then $L$ is the largest subgroup of
$G$ containing $H$ as an epimorphic subgroup.

\item[{\rm (vi)}] If $G$ and $H$ are smooth, then so is $L$.

\end{enumerate}

\end{theorem}

This result is proved in Section \ref{sec:aff}. In the setting of smooth
affine algebraic groups over an algebraically closed field, it gives back 
a statement of Bien and Borel (see \cite[Prop.~1]{Bien-Borel-I}), 
proved by Grosshans in \cite[\S 2, \S 23]{Grosshans}; our proof, based
on a descent argument, is somewhat more direct. 

Also, Theorem \ref{thm:aff} gives back most of the affinization theorem
for an arbitrary algebraic group $G$. More specifically, taking
for $H$ the trivial subgroup and using the fact that every quasi-affine
algebraic group is affine (see e.g.~\cite[VIB.11.11]{SGA3}), we obtain 
that $G$ has a smallest algebraic subgroup $L$ such that $G/L$ is affine, 
and $L$ is the largest anti-affine subgroup of $G$; moreover, $L$ is
connected. But the smoothness property of anti-affine algebraic groups 
does not extend to homogeneous spaces, as shown by 
Example \ref{ex:ns} at the end of Section \ref{sec:aff}.

Returning to the description of all the epimorphic subgroups $H$ 
of a smooth algebraic group $G$ over a field $k$, we may assume 
(by Theorem \ref{thm:main}, Corollary \ref{cor:aff} and Theorem
 \ref{thm:smooth}) $G$ to be affine and connected, 
$H$ smooth and connected, and $k$ algebraically closed; this is 
precisely the setting of \cite{Bien-Borel-I, Bien-Borel-II}. 
Even so, the structure of epimorphic subgroups is only partially
understood; a geometric criterion of epimorphicity is obtained in 
\cite{Petukhov} when $G$ is semi-simple and $k$ has characteristic $0$. 

The classification of epimorphic subgroups of non-smooth algebraic 
groups presents further open problems; see Example \ref{ex:ns}
again for a construction of such subgroups $H \subset G$, for which 
the quotient $G/H$ is non-smooth as well. 
Examples with a smooth quotient may be obtained as follows:
over any algebraically closed field $k$ of prime characteristic, 
there exist rational homogeneous projective varieties $X$ such 
that the automorphism group scheme $\Aut_X$ is non-smooth (see
\cite[Prop.~4.3.4]{Brion-Samuel-Uma}). Let then $G$ denote the neutral 
component of $\Aut_X$, and $H$ the stabilizer of a $k$-rational point
$x \in X$. Then $G$ is affine, non-smooth, and $X \cong G/H$; 
as a consequence, $H$ is non-smooth as well. Also, $H$ is epimorphic 
in $G$, by Theorem \ref{thm:main} or a direct argument as for parabolic 
subgroups. Thus, the description of epimorphic subgroups of possibly 
non-smooth algebraic groups entails that of automorphism group schemes 
of rational homogeneous projective varieties, which seems to be 
completely unexplored. 

\medskip

\noindent
{\bf Notation and conventions}. 
We use the books \cite{Demazure-Gabriel} and \cite{SGA3} as general 
references, and the expository text \cite{Brion} for some further results.

Throughout this note, we consider schemes over a fixed field $k$.
By an \emph{algebraic group}, we mean a group scheme $G$ of finite 
type over $k$; we denote by $e = e_G \in G(k)$ the neutral element,
and by $G^0$ the neutral component of $G$. The group law of $G$ will 
be denoted multiplicatively: $(x,y) \mapsto xy$. 

By a \emph{subgroup} of $G$, we mean a $k$-subgroup scheme $H$;
then $H$ is closed in $G$. \emph{Morphisms} of algebraic groups are 
understood to be homomorphisms of $k$-group schemes. 

Given a subgroup $H \subset G$ and a normal subgroup 
$N \triangleleft G$, we denote by $N \rtimes H$ the corresponding
semi-direct product, and by $N \cdot H$ the scheme-theoretic image
of the morphism 
\[ N \rtimes H \longrightarrow G, \quad (x,y) \longmapsto xy. \]
Then $N \cdot H$ is a subgroup of $G$, and the natural map 
$H/N \cap H \to G/N$ is a closed immersion with image $N \cdot H/N$
(see e.g.~\cite[VIIA.5.3.3]{SGA3}).

\section{Proof of Theorem \ref{thm:BB}}
\label{sec:BB}

(i) $\Rightarrow$ (ii): The action of $G$ on itself by right
multiplication yields a $G$-module structure on the algebra $\cO(G)$ 
(see \cite[Ex.~II.2.1.2]{Demazure-Gabriel}). Moreover, for any
subgroup $K \subset G$ acting on $\cO(G)$ by right multiplication, 
the natural map $\cO(G/K) \to \cO(G)^K$ is an isomorphism, as follows 
e.g.~from \cite[Cor.~VIA.3.3.3]{SGA3}. 
In particular, $\cO(G/H) \cong \cO(G)^H$ and $\cO(G)^G = k$. Thus, 
it suffices to show that every $f \in \cO(G)^H$ is fixed by $G$.

Consider the action of $G$ on itself by left multiplication; this
yields another $G$-module structure on $\cO(G)$, and $\cO(G)^H$ is
a $G$-submodule. By \cite[II.2.3.1]{Demazure-Gabriel}, 
we may choose a finite-dimensional $G$-submodule $V \subset \cO(G)^H$ 
that contains $f$. View $V$ as a vector group (the spectrum of the symmetric
algebra of the dual vector space) equipped with a compatible $G$-action, 
and consider the semi-direct product $G' := V \rtimes G$. 
Then $G'$ is an affine algebraic group; moreover, the maps 
\[ f_1: G \longrightarrow G', \quad g \longmapsto (0, g),
\qquad f_2: G \longrightarrow G',  \quad g \longmapsto (g \cdot f - f, g) \]
are two morphisms which coincide on $H$. Thus, $f_1 = f_2$, that is,
$f$ is fixed by $G$.

(ii) $\Rightarrow$ (iii): Recall from \cite[I.3.2]{Jantzen} that $V^H$ 
is the subscheme of $V$ associated with a linear subspace. So it suffices 
to show that every $v \in V^H(k)$ is fixed by $G$. Let $f \in \cO(V)$. 
Then the assignment $g \mapsto f(g \cdot v)$ defines 
$f_v \in \cO(G)^H \cong \cO(G/H) = k$. Thus, we have 
$f_v(g) = f_v(e)$, that is, $f(g \cdot v) = f(v)$ identically on $G$. 
Since this holds for all $f \in \cO(V)$, it follows that
$g \cdot v = v$ identically on $G$.

(iii) $\Rightarrow$ (iv): Let $\pi : V \to V_1$ denote the projection 
with kernel $V_2$. Consider the action of $G$ on $\End(V)$ by conjugation;
then $\End(V)$ is a finite-dimensional $G$-module, and 
$\pi \in \End(V)^H$. Thus, $\pi \in \End(V)^G$. It follows that
$V_1$ (the image of $\pi$) is normalized by $G$. Likewise, 
$V_2$ is normalized by $G$.

(iv) $\Rightarrow$ (i): Let $G'$ be an affine algebraic group, and
\[ f_1,f_2 : G \longrightarrow G' \] 
two morphisms which coincide on $H$. We may view 
$G'$ as a subgroup of $\GL(V)$ for some finite-dimensional 
vector space $V$ (see \cite[II.2.3.3]{Demazure-Gabriel}). 
This yields two linear representations 
\[ \rho_1, \rho_2 : G \longrightarrow \GL(V) \] 
which coincide on $H$. Consider the morphism 
\[ \rho_1 \oplus \rho_2 : G \longrightarrow \GL(V \oplus V). \]
Then we have with an obvious notation:
\[ V \oplus V = (V \oplus 0) \oplus \diag(V), \] 
where $V \oplus 0$ is normalized by $G$, and $\diag(V)$ is normalized
by $H$ (as $\rho_1 \vert_H = \rho_2 \vert_H$). So $\diag(V)$ is 
normalized by $G$, that is, $\rho_1 = \rho_2$. Thus, $f_1 = f_2$.

\section{Some auxiliary results}
\label{sec:aux}

Throughout this section, $G$ denotes an algebraic group, and 
$H \subset G$ a subgroup. We begin with a series of easy observations.

\begin{lemma}\label{lem:quot}
Assume that $H$ is epimorphic in $G$.

\begin{enumerate}

\item[{\rm (i)}] If $K \subset G$ is a subgroup containing $H$,
then $K$ is epimorphic in $G$.  

\item[{\rm (ii)}] If $N \triangleleft G$ is a normal subgroup, then 
$H/N \cap H$ is epimorphic in $G/N$.

\end{enumerate}

\end{lemma}

\begin{proof}
The assertion (i) is obvious, and implies that $N \cdot H$ is epimorphic
in $G$. As a direct consequence, $N \cdot H/N$ is epimorphic in $G/N$;
this yields the assertion (ii).
\end{proof}

\begin{lemma}\label{lem:conn}
Let $H$ be an epimorphic subgroup of $G$. 

\begin{enumerate}

\item[{\rm (i)}] If $G$ is finite, then $G = H$.

\item[{\rm (ii)}] For an arbitrary $G$, we have $G = G^0 \cdot H$. 

\end{enumerate}

\end{lemma}

\begin{proof}
(i) Since $G$ is affine and $H \subset G$ is affine epimorphic, we have 
$\cO(G/H) = k$ by Theorem \ref{thm:BB}. As the scheme $G/H$ is finite
and contains a $k$-rational point $x$, it follows that this scheme 
consists of the point $x$, hence $H = G$.

(ii) By Lemma \ref{lem:quot} (ii), $G^0 \cdot H/G^0$ is epimorphic 
in $G/G^0$. Thus, we may replace $G$ with $G/G^0$, and hence assume 
that $G$ is finite and \'etale. Then $H = G$ by (i).
\end{proof}

\begin{remark}\label{rem:aff}
(i) For finite \'etale groups, Lemma \ref{lem:conn} (i) also follows by 
adapting the proof of the surjectivity of epimorphisms of abstract groups, 
given in \cite{Linderholm}. 

\smallskip

\noindent
(ii) Lemmas \ref{lem:quot} and \ref{lem:conn} also hold in the category 
of affine algebraic groups, with the same proofs.
\end{remark}

\begin{lemma}\label{lem:ant}
Let $N \triangleleft G$ be a normal subgroup. If $H \supset G_{\ant}$, 
then $H/N \cap H \supset (G/N)_{\ant}$. Conversely, if
$H/N \cap H \supset (G/N)_{\ant}$ and $N$ is affine, then 
$H \supset G_{\ant}$.
\end{lemma}

\begin{proof}
By \cite[Lem.~3.3.6]{Brion}, the natural map 
$G_{\ant}/N \cap G_{\ant} \to (G/N)_{\ant}$ is an isomorphism. 
This yields the first assertion.

Conversely, assume that $H/N \cap H \supset (G/N)_{\ant}$; equivalently,
$(H/N \cap H)_{\ant} = (G/N)_{\ant}$. Using
\cite[Lem.~3.3.6]{Brion} again, it follows that 
$G_{\ant} \subset N \cdot H_{\ant}$. Thus, it suffices to show that 
$(N \cdot H)_{\ant} = H_{\ant}$. Using once more
\cite[Lem.~3.3.6]{Brion}, it suffices in turn to check that 
$(N \rtimes H)_{\ant} = H_{\ant}$. Since $N$ is affine and 
$N \rtimes H \cong N \times H$ as schemes, the affinization morphism 
\[ \varphi_{N \rtimes H} : N \rtimes H \longrightarrow 
\Spec \, \cO(N \rtimes H) \]
is identified with 
\[ \id \times \varphi_H : N \times H \longrightarrow 
N \times \Spec \, \cO(H). \] 
Taking fibers at $e$ yields the desired equality.
\end{proof}

Next, we obtain a result of independent interest, which generalizes 
(and builds on) Lemma \ref{lem:conn} (i):

\begin{lemma}\label{lem:proper}
If $G$ is proper and $H$ is epimorphic in $G$, then $H = G$.
\end{lemma}

\begin{proof}
The largest anti-affine subgroup $G_{\ant}$ is smooth, connected and proper,
that is, an abelian variety. Moreover, the quotient group $G/G_{\ant}$ 
is proper and affine, hence finite. Thus, using Lemma \ref{lem:quot} (ii) 
and Lemma \ref{lem:conn} (i), it suffices to show that $H$ contains 
$G_{\ant}$.

We now reduce to the case where \emph{$G$ and $H$ are smooth}.
For this, we may assume that $k$ has prime characteristic $p$.
Denote by $G_n$ (resp. $H_n$) the kernel of the $n$th relative 
Frobenius morphism of $G$ (resp. $H$). Then $G_n$ and $H_n$ are 
infinitesimal; also, $G/G_n$ and $H/H_n$ are smooth for $n \gg 0$ 
(see \cite[VIIA.8.3]{SGA3}). Using Lemma \ref{lem:quot} (ii) 
again together with Lemma \ref{lem:ant}, we see that it suffices 
to show that $H/H_n = H/G_n \cap H$ contains $(G/G_n)_{\ant}$. 
This yields the desired reduction.

Under this smoothness assumption, $G^0 = G_\ant$ is an abelian variety. 
Also, we have $G = G^0 \cdot H$ by Lemma \ref{lem:conn} (ii). Thus, 
$G^0 \cap H$ is centralized by $G^0$ and normalized by $H$, and hence 
is a normal subgroup of $G$. Using Lemma \ref{lem:quot} (ii) again, 
we may replace $G$, resp. $H$ with $G/G^0 \cap H$, resp. $H/G^0 \cap H$, 
and hence assume in addition that \emph{$G^0 \cap H$ is trivial}.

Under these assumptions, we may identify $G$ with $G^0 \rtimes H$.
Consider the diagonal action of $H$ on $G^0 \times G^0$ and form
the semi-direct product $G' := (G^0 \times G^0) \rtimes H$. Then
the maps
\[ f_1 : G \longrightarrow G', \quad (x,y) \longmapsto (x,e,y), \] 
\[ f_2 : G \longrightarrow G', \quad (x,y) \longmapsto (x,x,y), \] 
are two morphisms which coincide on $H$. Thus, $f_1 = f_2$. But then
$G^0$ must be trivial. 
\end{proof}

\section{Proof of Theorem \ref{thm:main}}
\label{sec:main}

(i) $\Rightarrow$ (ii): By \cite[Thm.~2]{Brion}, $G$ has 
a smallest normal subgroup $N$ such that $G/N$ is proper; moreover,
$N$ is affine. If $H$ is epimorphic in $G$, then $H/H \cap N$
is epimorphic in $G/N$ by Lemma \ref{lem:quot} (ii). Using Lemma
\ref{lem:proper}, it follows that $H/H \cap N =G/N$. So 
$H \supset G_{\ant}$ by Lemma \ref{lem:ant}. Thus, $\bar{H} := H/G_{\ant}$ 
is epimorphic in $\bar{G} := G/G_{\ant}$ by Lemma \ref{lem:quot} (ii) 
again. In view of Theorem \ref{thm:BB}, this yields
$\cO(\bar{G}/\bar{H}) = k$. As 
\[ \cO(\bar{G}/\bar{H}) \cong \cO(\bar{G})^{\bar{H}} 
= \cO(G/G_{\ant})^H \cong \cO(G/H), \]
we obtain $\cO(G/H) = k$.

(ii) $\Rightarrow$ (i): Let again $\bar{G}:= G/G_{\ant}$ and 
$\bar{H}:=H/G_{\ant}$. Then $\cO(\bar{G}/\bar{H}) = k$ by the above
argument. Using Theorem \ref{thm:BB}, it follows that $\bar{H}$ is
affine epimorphic in $\bar{G}$. Together with Lemma \ref{lem:conn}
(ii) and Remark \ref{rem:aff} (ii),
this yields $\bar{G} = \bar{G}^0 \cdot \bar{H}$, and hence
$\cO(\bar{G}/\bar{H}) \cong \cO(\bar{G}^0/\bar{G}^0 \cap \bar{H})$.
By Theorem \ref{thm:BB} again, it follows that $\bar{G}^0 \cap \bar{H}$ 
is affine epimorphic in $\bar{G}^0$. Also, note that $G = G^0 \cdot H$,
since $G_{\ant}$ is connected and contained in $H$.

Let $f_1,f_2: G \to G'$ be morphisms that coincide on $H$. Then
$f_1,f_2$ pull back to morphisms $f_1^0,f_2^0: G^0 \to G'^0$
which coincide on $G_{\ant} \triangleleft G^0 \cap H$. Moreover, the 
common scheme-theoretic image of $G_{\ant}$ under $f_1^0, f_2^0$ 
is contained in $G'_{\ant} \triangleleft G'^0$. This yields morphisms 
of affine algebraic groups
\[ \bar{f}_1^0,\bar{f}_2^0: \bar{G}^0 \to G'^0/G'_{\ant} \]
which coincide on $\bar{G}^0 \cap \bar{H}$. Thus, 
$\bar{f}_1^0 = \bar{f}_2^0$, that is, the morphism of schemes
\[ \varphi : G^0 \longrightarrow G'^0, 
\quad x \longmapsto f_1(x) f_2(x)^{-1} \]
factors through $G'_{\ant}$. We have 
\[ \varphi(x \, y) = f_1(x) \, f_1(y) \, f_2(y)^{-1} \, f_2(x)^{-1} \]
identically on $G^0 \times G^0$. Since $G'_{\ant}$ is contained
in the center of $G'^0$, it follows that $\varphi$ is a morphism 
of algebraic groups.

As $f_1$ and $f_2$ coincide on $G_{\ant} \subset H$, the kernel of 
$\varphi$ contains $G_{\ant}$. Thus, $\varphi$ factors through a morphism
of algebraic groups $\psi: \bar{G}^0 \to G'_{\ant}$. Since $\bar{G}^0$ 
is affine, so is the scheme-theoretic image of $\psi$. Also, $\psi$
is trivial on $\bar{G}^0 \cap \bar{H}$, an affine epimorphic subgroup
of $\bar{G}^0$. Thus, $\psi$ is trivial, that is, $f_1$ and $f_2$ 
coincide on $G^0$. Since these morphisms also coincide on $H$,
and $G = G^0 \cdot H$, we conclude that $f_1 = f_2$.

\section{Proof of Theorem \ref{thm:smooth}}
\label{sec:smooth}

(i) $\Rightarrow$ (ii): Recall from Theorem \ref{thm:main} that
$G_{\ant} \subset H$ and $\cO(G/H) = k$. Since $G_{\ant}$ is smooth,
it is contained in $H_{\red}$. Thus, using Theorem \ref{thm:main} again,
it suffices to show that $\cO(G/H_{\red}) = k$.

The natural map $u: G/H_{\red} \to G/H$ lies in a commutative square
\[ 
\xymatrix{
G \times H/H_{\red} \ar[r]^-{p_1} \ar[d]_{m} & G \ar[d]^{q} \\
G/H_{\red} \ar[r]^{u} & G/H,\\}
\]
where $p_1$ denotes the projection, $q$ the quotient map, and $m$ 
the pull-back of the action map $G \times G/H_{\red} \to G/H_{\red}$.
In fact, this square is cartesian and consists of faithfully flat morphisms
(see e.g.~the proof of \cite[Prop. 2.8.4]{Brion}).
As the scheme $H/H_{\red}$ is finite and has a unique $k$-rational
point, the map $p_1$ is finite and purely inseparable; thus, so
is $u$ by faithfully flat descent. Also, $G/H_{\red}$ and $G/H$ are smooth, 
since so is $G$. Thus, the induced map on rings of rational functions
\[ u^{\#}: k(G/H) \longrightarrow k(G/H_{\red}) \]
is injective, and there exists a positive integer $n$ (a power of 
the characteristic exponent of $k$) such that 
\[ k(G/H_{\red})^n \subset u^{\#} k(G/H). \] 
Also, by normality of $G/H$, we have 
$u^{\#} \cO(G/H) = u^{\#} k(G/H) \cap \cO(G/H_{\red})$
and hence 
\[ \cO(G/H_{\red})^n \subset u^{\#} \cO(G/H). \]
Since $\cO(G/H) = k$ and $\cO(G/H_{\red})$ has no non-zero nilpotents,
this yields the desired assertion.
  
(ii) $\Rightarrow$ (iii): We may replace $H$ with $H_{\red}$, 
and hence assume that $H$ is smooth. By Lemma \ref{lem:conn} (ii), 
we have $G = G^0 \cdot H$; thus, the natural map $H/H^0 \to G/G^0$
is surjective. Also, $G_{\ant}$ is connected, and contained in $H$
by Theorem \ref{thm:main}; hence $G_{\ant} \subset H^0$. So, using 
Theorem \ref{thm:main} once more, we are reduced to checking that 
$\cO(G^0/H^0) = k$.

Note that 
\[ k = \cO(G/H) = \cO(G^0 \cdot H/H) \cong \cO(G^0/G^0 \cap H). \]
Next, consider the natural map 
\[ \psi : G^0/H^0 \longrightarrow G^0/G^0 \cap H. \]
The finite \'etale group 
$F := (G^0 \cap H)/H^0 \subset H/H^0$ acts on the right on $G^0/H^0$, 
and $\psi$ is the categorical quotient for that action.
Thus, $\cO(G^0/H^0)^F \cong \cO(G^0/G^0 \cap H)$, 
and hence the algebra $\cO(G^0/H^0)$ is integral over 
$\cO(G^0/G^0 \cap H) = k$. As above, this implies the desired assertion.

(iii) $\Rightarrow$ (i): This follows by reverting some of the
previous arguments. More specifically, we have
$G_{\ant} = (G^0)_{\ant} \subset H^0_{\red} \subset H$. Also, 
$G = G^0 \cdot H_{\red} = G^0 \cdot H$ and hence 
\[ \cO(G/H) \cong \cO(G^0/G^0 \cap H) \cong \cO(G^0)^{G^0 \cap H}
\subset \cO(G^0)^{H^0} = k. \]
Thus, $H$ is epimorphic in $G$ by Theorem \ref{thm:main} again.

\section{Proof of Theorem \ref{thm:aff}}
\label{sec:aff}

(i) Consider the action of $G$ on $\cO(G)$ via right multiplication and 
let $L \subset G$ be the centralizer of the subspace 
$\cO(G)^H \subset \cO(G)$. In view of \cite[II.1.3.6]{Demazure-Gabriel},
$L$ is represented by a subgroup of $G$ that we will also denote by $L$.
Since $L$ acts trivially on $\cO(G)^H$, we have 
$\cO(G)^H \subset \cO(G)^L$. On the other hand, $H \subset L$
and hence $\cO(G)^L \subset \cO(G)^H$. Thus, $\cO(G)^L = \cO(G)^H$. 

We show that there exists a finite subset $F \subset \cO(G)^H$ such 
that $L$ is the centralizer $C_G(F)$. Indeed,  we may find $F$
such that $C_G(F)$ is minimal among all such centralizers. Then
$C_G(F \cup \{ f \}) = C_G(F)$ for any $f \in \cO(G)^H$, and hence
$C_G(F)$ centralizes the whole subspace $\cO(G)^H$.

Choose $F = \{ f_1, \ldots, f_n \} \subset \cO(G)^H$ such that
$L = C_G(F)$. Then $L$ is the centralizer in $G$ of 
$f_1 + \cdots + f_n$, viewed as a $k$-rational point of  
$\cO(G) \oplus \cdots \oplus \cO(G) =: n \cO(G)$.
As $f_1, \ldots, f_n$ are contained in some finite-dimensional 
$G$-submodule $V \subset n \cO(G)$, it follows that $G/L$ is isomorphic 
to a subscheme of the affine space associated with $V$ 
(see \cite[III.3.5.2]{Demazure-Gabriel}). In view of \cite[II.5.1.2]{EGA},
it follows that $G/L$ is quasi-affine. In other terms, the affinization map 
$\varphi_{G/L }$ is an open immersion. Since 
$\cO(G/L) = \cO(G/H)$, this yields the desired commutative triangle
\[ 
\xymatrix{
G/H \ar[d]_{u} \ar[dr]^{\varphi_{G/H}} \\
G/L \ar[r]_{\varphi_{G/L}} & X, \\}
\]
where $u$ denotes the natural map, and 
$X = \Spec \, \cO(G/H) = \Spec \, \cO(G/L)$.

Let $K$ be a subgroup of $G$ such that $K \supset H$ and $G/K$ 
is quasi-affine. Then we have a commutative square of $G$-equivariant
morphisms
\[ 
\xymatrix{
G/H \ar[r]^-{\varphi_{G/H}} \ar[d]_{v} & \Spec \, \cO(G/H) = \Spec \, \cO(G/L) 
\ar[d]^{ \varphi_v} \\
G/K \ar[r]^-{\varphi_{G/K}} & \Spec \, \cO(G/K),\\}
\]
where $\varphi_{G/K}$ is an open immersion. Thus, $v$ factors through $u$,
and hence $L \subset K$. 

(ii) In view of \cite[I.1.2.6]{Demazure-Gabriel}, the formation of 
the affinization morphism commutes with arbitrary field extensions.
Thus, so does the formation of $L$.

(iii) Consider a subgroup $K$ of $G$ containing $H$ such that $K/H$
is anti-affine. Denote by $q: G \to G/H$ the quotient map and by
$x = q(e_G)$ the base point. Then the pull-back map 
$\cO(G)^H \to \cO(K)^H \cong \cO(K/H) = k$ is identified with the
homomorphism $\cO(G/H) \to k$ given by evaluation at $x$. 
Thus, $K/H \subset G/H$ is contained in the fiber of 
$\varphi_{G/H}$ at $x$. By (i), this fiber is $L/H \subset G/H$. It
follows that $K \subset L$.

We now show that $L/H$ is anti-affine. As in the proof of Theorem
\ref{thm:smooth}, we have a cartesian diagram of faithfully flat
morphisms
\[ 
\xymatrix{
G \times L/H \ar[r]^-{p_1} \ar[d]_{m} & G \ar[d]^{r} \\
G/H \ar[r]^{u} & G/L,\\}
\]
where $p_1$ denotes the projection, $r$ the quotient map,
and $m$ the pull-back of the action map $G \times G/H \to G/H$. 
Thus, we obtain a canonical isomorphism of sheaves on $G$:
\[ r^*(u_* \cO_{G/H}) \stackrel{\cong}{\longrightarrow} 
(p_1)_*(m^* \cO_{G/H}).\]
Clearly, we have  
$m^*\cO_{G/H} = \cO_{G \times L/H}$ and $r^* \cO_{G/L} = \cO_G$. 
Moreover, the natural map $\cO_{G/L} \to u_*\cO_{G/H}$ 
is an isomorphism, since $\cO(G/L) = \cO(G/H)$ and $G/L$
is covered by open affine subschemes of the form
$(G/L)_f$, where $f \in \cO(G/L)$ (see 
\cite[II.5.1.2]{EGA}). It follows that the natural map 
$\cO_G \to (p_1)_* \cO_{G \times L/H}$
is an isomorphism as well. In particular, this yields
$\cO(G) = \cO(G \times L/H)$, and hence 
$\cO(L/H) = k$ as desired.

(iv) It suffices to show that the natural map
$L^0/L^0 \cap H \to L/H$ is an isomorphism, as every 
homogeneous space under a connected algebraic group is
geometrically irreducible (see e.g.~\cite[VIA.2.6.6]{SGA3}).
The quotient $L/L^0 \cdot H$ is finite and \'etale (since so
is $L/L^0$), and anti-affine (since so is $L/H$). Thus, this
quotient consists of a unique $k$-rational point. Hence
$L = L^0 \cdot H$; this yields the desired assertion.

(v) Let $K$ be a subgroup of $G$ containing $H$. As $K$ is
affine, we have by Theorem \ref{thm:main} that $K/H$ is 
anti-affine if and only if $H$ is epimorphic in $K$.
In view of (ii), this yields the assertion.

(vi) By (ii), we may assume that $k$ is algebraically closed.
Then $H \subset L_{\red} \subset L$ and the natural map
$G/L_{\red} \to G/L$ is finite, as shown in the proof of Theorem 
\ref{thm:smooth}. Since $G/L$ is quasi-affine, so is $G/L_{\red}$ in
view of \cite[II.5.1.2, II.5.1.12]{EGA}. Thus, $L = L_{\red}$
by the minimality of $L$, i.e., $L$ is smooth.

\begin{example}\label{ex:ns}
Assume that $k$ has characteristic $p > 0$. Let $Y = G/H$ be 
a smooth anti-affine homogeneous space, where $G$ is affine and
$H \subsetneq G$. We will construct a non-smooth anti-affine homogeneous 
space $X$ under an algebraic group containing $G$, such that $X$
contains $Y$ as its largest smooth subscheme. For this, we use a process
of ``infinitesimal thickening'' of an arbitrary homogeneous space $G/H$.
 
Let $M$ be a finite-dimensional $G$-module. Viewing $M$ as a $p$-Lie 
algebra with zero bracket and $p$th power map, we obtain 
a commutative infinitesimal algebraic group $G_p(M)$ of height $1$ 
(see \cite[VIIA.8.1.2]{SGA3}). The action of $G$ on $M$ yields
an action on $G_p(M)$ by automorphisms of algebraic groups; we denote by
$G_p(M) \rtimes G$ the corresponding semi-direct product.

Next, let $N \subset M$ be an $H$-submodule. As above, we may form the
semi-direct product $G_p(N) \rtimes H$; this is a subgroup of 
$G_p(M) \rtimes G$. Consider the homogeneous space
\[ X :=  G_p(M) \rtimes G / G_p(N) \rtimes H. \]
The chain of inclusions 
$ G_p(N) \rtimes H \subset  G_p(N) \rtimes G  \subset G_p(M) \rtimes G$
yields a morphism
\[ f: X \longrightarrow G_p(M) \rtimes G / G_p(M) \rtimes H \cong G/H = Y. \]
Moreover, $f$ is $G$-equivariant and its fiber at the base point 
$y \in Y(k)$ is $H$-equivariantly isomorphic to $G_p(M)/G_p(N)$. 
The latter quotient is canonically isomorphic to $G_p(M/N)$, by
\cite[VIIA.8.1.3]{SGA3}. The neutral element of $G_p(M/N)$ is fixed 
by $H$, and hence yields a section $s : Y \to X$ of $f: X \to Y$. 
As $G_p(M/N)$ is infinitesimal, $f$ and $s$ induce mutually inverse 
homeomorphisms of the underlying topological spaces of $X$ and $Y$.

We have an isomorphism
\[ \cO(X) \cong (\cO(G) \otimes \cO(G_p(M/N)))^H, \]
where $H$ acts simultaneously on $\cO(G)$ by left multiplication,
and on $\cO(G_p(M/N))$ via its linear action on $M/N$. Also, recall
from \cite[VIIA.7.4]{SGA3} the canonical isomorphism
\[ \cO(G_p(M/N)) \cong \Sym(M/N)^*/I, \]
where $\Sym(M/N)^*$ denotes the symmetric algebra of the dual
module of $M/N$, and $I$ the ideal generated by the $p$th powers of
all elements of $(M/N)^*$.

Assume that $G$ is affine. By a theorem of Chevalley
(see e.g.~\cite[II.2.3.5]{Demazure-Gabriel}), we may choose a
finite-dimensional $G$-module $M$ and a hyperplane $N \subset M$
such that $H$ is the stabilizer of $N$ for the $G$-action on $M$. 
In particular, $N$ is an $H$-submodule of $M$; we denote by $L = M/N$ 
the quotient line. Then we have an isomorphism of $H$-modules
\[ \cO(G_p(M/N)) \cong \bigoplus_{i = 0}^{p-1} L^{-i}, \]
where $L^{-i}$ denotes the $i$th tensor power of $L^*$ (in particular, 
$L^0$ is the trivial $H$-module $k$). Denoting by $\cL$ the $G$-linearized
invertible sheaf on $Y = G/H$ associated with the $H$-module $L$
(as in \cite[I.5.8]{Jantzen}), we then have
\[ \cO(X) \cong \bigoplus_{i = 0}^{p-1} (\cO(G) \otimes L^{-i})^H 
\cong \bigoplus_{i = 0}^{p-1} \Gamma(Y, \cL^{-i}). \]

Assume in addition that $Y$ is smooth, anti-affine and non-trivial.
Then the section $s$ identifies $Y$ to the largest smooth subscheme
of $X$. It remains to check that $X$ is anti-affine; for this, we
show that $\Gamma(Y, \cL^{-i}) = 0$ for all $i \geq 1$. Consider 
$\Gamma(Y, \cL) = (\cO(G) \otimes L)^H$. The exact sequence of
$H$-modules $0 \to N \to M \to L \to 0$ yields a morphism of $G$-modules
$(\cO(G) \otimes M)^H \to \Gamma(Y,\cL)$. Moreover, we have
an isomorphism of $G$-modules
$(\cO(G) \otimes M)^H \cong \cO(G/H) \otimes M = M$
in view of \cite[I.3.6]{Jantzen}. This defines a morphism of $G$-modules
$\varphi : M \to \Gamma(Y,\cL)$, dual to the immersion of 
$Y$ into the projective space of hyperplanes in $M$. 
In particular, $\varphi(N)$ is non-zero and consists of sections 
$\sigma \in \Gamma(Y,\cL)$ that vanish at the base point $y$, 
i.e., $\sigma_y \in \fm_y \cL_y$. Choose such a section $\sigma \neq 0$ 
and let $\tau \in \Gamma(Y,\cL^{-i})$. Then 
$\sigma^i \tau \in \Gamma(Y,\cO_Y) = k$ and $\sigma^i \tau$
vanishes at $y$ as well. Thus, $\sigma^i \tau = 0$, hence
$\tau = 0$ as $Y$ is smooth and geometrically irreducible.
\end{example}

\medskip

\noindent
{\bf Acknowledgments.} Many thanks for the anonymous referee for his/her
careful reading and very helpful suggestions.

\bibliographystyle{amsalpha}

\end{document}